\documentclass[article, 12pt]{amsart}

\usepackage[T1]{fontenc}

\usepackage[margin=1in]{geometry}

\usepackage{amscd}

\usepackage{amsmath}

\usepackage{amsfonts}

\usepackage{amssymb}

\usepackage{amsthm}

\usepackage{arydshln}

\usepackage{fancyhdr}

\usepackage{verbatim}

\usepackage{tikz-cd}

\usepackage[all]{xy}

\usepackage{color}

\usepackage[colorlinks=true, linkcolor=blue, citecolor=blue,
pagebackref=true]{hyperref}

\usepackage{fouriernc}

\usepackage{mathtools} \usepackage{etoolbox}
\patchcmd{\section}{\scshape}{\bfseries}{}{} \makeatletter
\renewcommand{\@secnumfont}{\bfseries} \makeatother

\theoremstyle{definition}

\theoremstyle{plain} \newtheorem{theorem}{Theorem}[section]

\theoremstyle{plain} 

\theoremstyle{plain} \newtheorem{proposition}[theorem]{Proposition}

\theoremstyle{plain} 

\theoremstyle{plain} \newtheorem{corollary}[theorem]{Corollary}

\theoremstyle{remark} \newtheorem*{remark}{Remark}

\theoremstyle{definition} 

\theoremstyle{definition} \newtheorem*{definition*}{Definition}

\theoremstyle{definition} 

\theoremstyle{remark} 

\makeatletter \renewenvironment{proof}[1][\proofname]
{\par\pushQED{\qed}\normalfont\topsep6\p@\@plus6\p@\relax\trivlist\item[\hskip\labelsep\bfseries#1\@addpunct{.}]\ignorespaces}{\popQED\endtrivlist\@endpefalse}
\makeatother

\newcommand{\RR}{\mathbb{R}}

\newcommand{\QQ}{\mathbb{Q}}

\newcommand{\NN}{\mathbb{N}}

\newcommand{\ZZ}{\mathbb{Z}}

\newcommand{\calQ}{\mathcal{Q}}

\newcommand{\diam}{\text{diam}}

\newcommand{\dimH}{\dim_{\textup H}}

\renewcommand{\leq}{\leqslant} \renewcommand{\geq}{\geqslant}

\DeclareMathOperator{\Leb}{Leb}
\DeclarePairedDelimiter{\abs}{\lvert}{\rvert}

\DeclarePairedDelimiter{\floor}{\lfloor}{\rfloor}

\DeclarePairedDelimiter{\set}{\lbrace}{\rbrace}
\DeclarePairedDelimiter{\parens}{\lparen}{\rparen}
\DeclarePairedDelimiter{\brackets}{\lbrack}{\rbrack}

\def\1int{{[0,1]}}

\theoremstyle{plain}
\newtheorem{thm}{Theorem}[section]

\newtheorem{prop}[thm]{Proposition}

\newtheorem{ques}[thm]{Question}

\title{Diophantine approximation in metric space}

\author{Jonathan M. Fraser} \address{Mathematical Institute, University of St Andrews, Scotland, KY16 9SS}
\email{jmf32@st-andrews.ac.uk}
\author{Henna Koivusalo} \address{School of Mathematics, University of Bristol,  BS8 1UG}
\email{henna.koivusalo@bristol.ac.uk}
\author{Felipe A.~Ram\'irez} \address{Weslyan University, Middletown, CT, USA}
\email{framirez@wesleyan.edu}
\thanks{The authors were supported by a Mathematisches Forschungsinstitut Oberwolfach - Research in Pairs Grant.  JMF was financially supported by an \emph{EPSRC Standard Grant} (EP/R015104/1) and  a  \emph{Leverhulme Trust Research Project Grant} (RPG-2019-034).}

\begin{document}

\maketitle

\begin{abstract}
  Diophantine approximation is traditionally the study of how well
  real numbers are approximated by rationals. We propose a model for
  studying Diophantine approximation in an arbitrary totally bounded
  metric space where the rationals are replaced with a countable
  hierarchy of ``well-spread'' points, which we refer to as
  \emph{abstract rationals}.  We prove various Jarn\'ik--Besicovitch
  type dimension bounds and investigate their sharpness.
\end{abstract}

\section{Introduction}

\subsection{Diophantine approximation in a metric space}

Let $(F,d)$ be a non-empty totally bounded   metric space.  For $q \in \mathbb{N}$, define  \emph{abstract rationals of level $q$} by fixing a maximal $(1/q)$-separated set $P(q) \subseteq F$.  By $(1/q)$-\emph{separated} we mean that for all $p_1, p_2 \in P(q)$, $d(p_1,p_2) \geq 1/q$ and by  \emph{maximal} we  mean it is impossible to add another point to the set and maintain the separation.  Since $F$ is totally bounded, $P(q)$ is a finite set for all $q$. One should think of $P(q)$ as the set of rationals with denominator $q$ inside $F$.  We refer to the set $\cup_q P(q)$ as a set of \emph{abstract rationals}, noting that it is immediately a countable  dense subset of $F$.  For $t >0$ let
\begin{eqnarray*}
F_t &=& \{ x \in F  : \text{ i.o. for some  $p \in P(q)$ }\, d(p,x) \leq 1/q^t \} \\
& =& \bigcap_N \bigcup_{q>N} \bigcup_{p \in P(q)} B(p , 1/q^t) 
\end{eqnarray*}
where $B(x,r)$ generally denotes the closed ball centred at $x \in X$ with radius $r>0$.  Clearly, if $t \leq 1$, then $F_t = F$ and so we are interested in $t>1$.  Computing (or estimating)  the size (e.g. Hausdorff dimension) of $F_t$ describes ``how much of $F$ can be far away from a chosen  set of well spread out  points''.  

Classical Diophantine approximation may be cast in this language.  For
example, if $F=[0,1]$ with the Euclidean distance and
$P(q)= \mathbb{Z}/q \cap [0,1]$ are the rationals with denominator
$q$, then $F_t$ is the set of $t$-well-approximable real numbers in
$[0,1]$.  In this setting, Dirichlet's theorem tells us that
$F_2 = [0,1]$, the Borel--Cantelli lemma tells us that $F_t$ is a
null set for all $t>2$, and the Jarn\'ik--Besicovitch theorem gives
\[
\dim_\textup{H} F_t = 2/t
\]
for all $t \geq 2$, where $\dim_\textup{H}$ denotes the Hausdorff
dimension. That is, the exponent $t_0=2$ is a natural
  threshold beyond which it is sensible to study the dimensions of
  $F_t$. Our aim here is to establish analogous results in the general
  setting of abstract rationals.

  The inspiration for studying abstract rationals comes from problems
  of \emph{intrinsic} Diophantine approximation.  In such problems,
  one has a subset $X$ of $\mathbb R^n$, and the goal is to understand
  how well its points can be approximated by rational points which
  also lie on $X$. The set $X$ can for example be taken to be a
manifold \cite{FKMS18} or a fractal set \cite{BFR11}. Questions on
intrinsic Diophantine approximation on fractals have attracted much
attention and there have been many interpretations and answers to the
now-classical Mahler's problem: `How well can irrationals in the
Cantor set be approximated by rationals in the Cantor set?'
Interpretations include random rational points \cite{BD16}, intrinsic
rationals \cite{FMS18, FS14, LSV07, S21}, and dynamical rationals
\cite{B19, AB}.

In Section \ref{sec:Main results} we provide analogues to the
Jarn\'ik--Besicovitch theorem in metric spaces using the language of
abstract rationals.  It turns out that the problem is subtle and the
answer depends on how the abstract rationals are chosen. Despite the
generality of the notion of abstract rationals, the bounds we prove in
this general set-up are sharp in many cases of particular
interest. For example, Theorem \ref{main} below demonstrates that the
upper bound for the dimension value in the Jarn\'ik--Besicovitch
theorem holds in much more generality, replacing ``$2$'' with
``$\dim_\textup{P}F + 1$'' where $\dim_\textup{P}$ denotes the
\emph{packing} dimension. We also show in Proposition \ref{sharpness1}
that in general one cannot replace ``$2$'' with
``$\dim_\textup{H}F + 1$'' or ``$2\dim_\textup{H}F$'', for example.
In order to estimate $\dim_\textup{H} F_t$ in general we naturally
encounter several different notions of fractal dimension, which we
recall in the following subsection.

We obtain stronger dimension bounds in the special case when $F$ is Ahlfors--David regular, see Theorem \ref{adthm}.  In Section \ref{sec:spectrum}, during the investigation of the
sharpness of the bounds in the Ahlfors--David regular case, we
encounter the notions of {\em Dirichlet exponent} and {\em Dirichlet
  spectrum}. By Dirichlet exponent we mean the threshold $t_0$ after
which $F_{t}$ does not (essentially) contain a ball. The
  definition is a natural extension of the classical Dirichlet
  exponent ($t_0=2$) to the setting of abstract rationals. We find
  that in general the Dirichlet exponent of $F$ depends on the choice
  of abstract rationals, so we are occasioned to consider the
  collection of possible Dirichlet exponents for a given $F$---the
  Dirichlet spectrum of $F$. The main result of Section~\ref{sec:Main
    results} and the mass transference principle of Beresnevich and
  Velani~\cite{BV06} allow us to bound the Dirichlet spectrum and
  reveal an interplay between Dirichlet exponents and the dimensions
  of $F_t$. In studying that interplay, we find that the dimension
  bounds from Theorem   \ref{adthm} are
  sharp.

\subsection{Dimension theory}

An important problem in the study of fractal objects is to characterise their dimension.  There are many ways to define dimension and the relationship between the notions is especially interesting, and relevant to us. See \cite{falconer, fraser} for more background on dimension theory and fractal geometry. 

Given $s \geq 0$,  the \emph{$s$-dimensional Hausdorff measure} of a set $X$ in a metric space is defined by
\[
\mathcal{H}^s (X) = \lim_{ r \to 0} \, \inf \bigg\{ \sum_{i } \lvert U_i \rvert^s : \{ U_i \}_{i} \text{ is a countable $ r$-cover of $X$} \bigg\}.
\]
The \emph{Hausdorff dimension} of $X$ is then 
\[
\dim_\textup{H} X = \inf \Big\{ s \geq 0: \mathcal{H}^s (X) =0 \Big\} = \sup \Big\{ s \geq 0: \mathcal{H}^s (X) = \infty \Big\}.
\] 
A subtle feature of the Hausdorff dimension is that sets of different sizes may be used in the covers and their contribution is weighted accordingly.  The (upper) box dimension is a simpler notion of dimension which avoids this subtlety.  The \emph{upper box dimension} of $X$ is
\[
\overline{\dim}_\textup{B} X = \limsup_{ r \to 0} \,  \frac{\log N_r (X)}{-\log  r},
\]
where $N_ r (X)$ is maximum cardinality of an $ r$-separated subset of $X$.  If the $\limsup$ is really a limit then we refer to the box dimension, written $\dim_\text{B} X$. Note that if $X$ is not totally bounded then $N_r(X)$ may be infinite. 

The  packing dimension of $X$ may be defined via packing measure in an analogous way to the Hausdorff dimension, but we use a common alternative formulation in terms of the upper box dimension.   We define the \emph{packing dimension} of $X$ by
\[
\dim_\textup{P} X= \inf \left\{ \sup_i \overline{\dim}_\textup{B}X_i \ : \ X=\bigcup_i X_i\right\}. 
\]
It should be clear from the definitions that
\[
\dim_\textup{H} X \leq \dim_\textup{P} X \leq \overline{\dim}_\textup{B} X .
\]
Finally, the \emph{lower dimension} of $X$ is defined by
\begin{eqnarray*} \index{lower dimension}
\dim_\textup{L} X \ = \  \sup \Bigg\{ \  s \geq 0 &:& \text{     there exists a constant $C >0$ such that,} \\
&\,& \hspace{10mm}  \text{for all $0<r<R \leq |X| $ and $x \in X$, } \\ 
&\,&\hspace{27mm}  \text{$ N_r\big( B(x,R) \cap X \big) \ \geq \ C \bigg(\frac{R}{r}\bigg)^s$ } \Bigg\}
\end{eqnarray*}
where $|X|$ denotes the diameter of $X$.  The lower dimension is the natural dual to the Assouad dimension and identifies the ``thinnest part'' of the space.  If $X$ is closed, then 
\[
\dim_\textup{L} X  \leq \dim_\textup{H} X.
\]
See \cite[Chapter 3]{fraser} for more background on the lower dimension, including its relevance in a different area of Diophantine approximation, e.g. \cite[Section 14.2]{fraser}. 

\section{Main results}\label{sec:Main results}

\subsection{General bounds}

Our first result provides general bounds for the dimension of $F_t$ which hold independently of how the abstract rationals are chosen.  The general lower bound assumes compactness of $F$.  If $F$ is not compact, then this lower bound does not hold. For example, if $F=[0,1]\cap \mathbb Q$, then the packing dimension of $F=F_1$ is $0$ and the lower dimension is $1$.

\begin{thm} \label{main}
For all $t \geq 1$
\[
\dim_\textup{H} F_t \leq \frac{\dim_\textup{P} F+1}{t}
\]
and, if $F$ is compact,
\[
\dim_\textup{P} F_t \geq \frac{\dim_\textup{L} F}{t}.
\]

\end{thm}

\begin{proof}
We begin with the general upper bound. Let $b>\dim_\textup{P} F$ and let 
\[
F= \bigcup_i F_i
\]
be a countable decomposition of $F$ with the property that $ \overline{\dim}_\textup{B} F_i <b$ for all $i$.  Temporarily fix an index $i$.  We write
\[
d(p,F_i) = \inf \{ d(p,x) : x \in F_i\}.
\]
For all $N \geq 1$
\[
\bigcup_{q>N} \bigcup_{\substack{p \in P(q): \\ d(p,F_i) \leq 1/q}}\{B(p , 1/q^t) \}
\]
provides a $2/N^t$-cover of $F_t \cap F_i$.  Therefore, for some constant $C_1 \geq 1$ depending on $F$, $i$ and  $b$,
\begin{eqnarray*}
\mathcal{H}^s_{2/N^t} (F_t \cap F_i) &\leq& \sum_{q>N} \sum_{\substack{p \in P(q): \\ d(p,F_i) \leq 1/q}} |B(p,1/q^t)|^s  \\ \\
&=&  \sum_{q>N} (\#\{ p \in P(q):d(p,F_i) \leq 1/q\} )(2/q^t)^s  \\ \\
&\leq & 2^s C_1 \sum_{q \geq 1} q^{b-ts} < \infty
\end{eqnarray*}
provided $s > \frac{b+1}{t}$.  For such $s$
\[
\mathcal{H}^s(F_t) \leq \sum_i  \mathcal{H}^s (F_t \cap F_i) = 0
\]
and the desired result follows.

Turning our attention to the general lower bound, let $0<l < \dim_\textup{L} F$ (we may assume the lower dimension is strictly positive). Fix a rapidly increasing sequence of integers $q_k$ satisfying
\begin{equation} \label{rapid}
q_1 q_2 \cdots q_k \leq \log(q_{k+1})
\end{equation}
 and construct a nested decreasing sequence of compact sets $E_k \subseteq F_t$ by setting $E_1 = B(z,q_1^{-t})$ with $z \in F$ chosen arbitrarily and then letting $E_k$ be the union of the closed balls $B(p,q_k^{-t})$ with $p \in P(q_k)$ which lie completely inside $E_{k-1}$.  The sets $E_k$ are compact since $F$ is compact and so by the  Cantor intersection theorem
\[
E = \bigcap_{k} E_k \subseteq F_t
\]
with $E$ non-empty and compact.  We show $\dim_\textup{P} E \geq l/t$, which completes the proof. Let $\mu$ be a probability measure supported on $E$ constructed by assigning unit mass to $E_1$ and then, given a $k$-level cylinder $B(p,q_k^{-t}) $ used in the construction of $E_k$, redistribute the mass of $B(p,q_k^{-t}) $ equally among the $(k+1)$-level cylinders lying inside.  Since the $q_k$ are rapidly increasing and $P(q_k)$ is $(1/q_k)$-separated, using the definition of lower dimension, we may assume that there are at least
\[
 \left( \frac{q_k^{-t}}{q_{k+1}^{-1}} \right)^l
\]
$(k+1)$-level cylinders  inside every $k$-level cylinder.  For each $k$, consider $r = q_k^{-t}$.  It follows that any ball $B(x,r)$ with $x \in E$ intersects at most one $k$-level cylinder.  Recall that the centres of the $k$-level cylinders are $(1/q_k)$-separated.  Therefore
\begin{eqnarray} \label{massestimate}
\mu(B(x,r)) &\leq&     \left( \frac{q_1^{-t}}{q_{2}^{-1}} \right)^{-l}  \left( \frac{q_2^{-t}}{q_{3}^{-1}} \right)^{-l} \cdots  \left( \frac{q_{k-1}^{-t}}{q_{k}^{-1}} \right)^{-l} \nonumber  \\  \nonumber \\
& = &  q_1^{tl}(q_2 \cdots q_{k-1})^{tl-l} (q_k^{-1})^{l} \\  \nonumber \\
&\leq &  \left(\log(q_k)\right) ^{tl} (q_k^{-1})^{l}  \qquad \text{by \eqref{rapid}}  \nonumber \\  \nonumber  \\
&= &   \left(\frac{|\log r|}{t}\right) ^{tl} r^{l/t} .  \nonumber 
\end{eqnarray}
This proves that for all $ x \in E$, the upper local dimension of $\mu$ at $x$ is at least $l/t$.  Applying standard estimates for packing dimension coming from local dimension of measures, e.g. \cite[Proposition 2.3]{techniques}, this proves that $\dim_\textup{P} E \geq l/t$ as required.
\end{proof}

\subsection{Sharpness of general bounds}

\subsubsection{Sharpness of upper bound}

Comparing the upper bound from Theorem \ref{main} with the classical
Jarn\'ik--Besicovitch theorem, one might ask if the packing dimension can be
replaced by the Hausdorff dimension, or indeed if an upper bound of
the form $c \dim_\textup{H} F/t$ is possible for some uniform constant
$c$.  The following proposition shows that this is not possible and
demonstrates that the packing dimension should appear in bounding the
Hausdorff dimension of $F_t$.

\begin{prop} \label{sharpness1}
For all $t > 2$, there exists a compact set $F \subseteq [0,1]$  and a choice of abstract rationals such that 
\[
\dim_\textup{P}F =1
\]
and
\[
\dim_\textup{H}F_t = \dim_\textup{H}F =   2/t = \frac{ \dim_\textup{P}F+1}{t}.
\]
In particular we cannot replace the packing dimension by the Hausdorff dimension in Theorem \ref{main}.
\end{prop}

\begin{proof}
Let $E_t$ denote the set of $t$-well-approximable numbers in $[0,1]$ in the classical setting.  Then
\[
 \dim_\textup{H}E_t =   2/t
\]
and $\mathcal{H}^{2/t}(E_t) = \infty$.  This is the Jarn\'ik--Besicovitch theorem, but we emphasise the infinitude of the Hausdorff measure which we need for what follows, see \cite[Theorem 11]{velanibook}.    By \cite[Theorem 4.10]{falconer} there exists a compact subset $E_t^* \subseteq E_t$ such that
\[
 \dim_\textup{H}E_t^* =   2/t.
\]
Let 
\[
F=E_t^* \cup \{ p/q : p \in \mathbb{Z}, \, q \in \mathbb{N}, \, |p/q-x| \leq 1/q \text{ for some } x \in E_t^*\}.
\]
Then $F$ is compact and 
\[
\dim_\textup{H}F = \dim_\textup{H}E_t =  2/t.
\]
The abstract rationals $P(q)$ are then chosen to be $P(q) = \mathbb{Z}/q \cap F$. By construction $E_t^* \subseteq E_t \cap F = F_t$ and so
\[
\dim_\textup{H}F_t = \dim_\textup{H}F =   2/t.
\]
Moreover, by Theorem \ref{main}, 
\[
2/t = \dim_\textup{H}F_t  \leq \frac{\dim_\textup{P}F+1}{t}
\]
which forces $\dim_\textup{P}F =1$, completing the proof.
\end{proof}

The next result is along similar lines but this time the construction works for arbitrary choices of abstract rationals.

\begin{prop}
For all integers $t \geq 2$, there exists a compact set $F \subseteq [0,1]$ such that 
\[
F_s = F
\]
for all $1 \leq s < t$ and all choices of abstract rationals $\mathcal{P} = \bigcup_qP(q)$.  In particular,  for all $t \geq 1$,
\[
\sup_F \inf_\mathcal{P}\frac{\dim_\textup{H} F_t}{\dim_\textup{H} F} = 1
\]
where the supremum is taken over compact sets $F \subseteq [0,1]$ with $\dim_\textup{H} F>0$ and the infimum is taken over all choices of abstract rationals for $F$.  
\end{prop}

\begin{proof}
Consider an iterated function system (IFS) type construction which alternates between $\{x \mapsto x/2, x \mapsto(x+1)/2\}$ and $\{x \mapsto x/2\}$. See \cite{falconer} for background on IFS theory.  Let $N(k)$ be the number of times we choose the first IFS in the first $k$ levels.  Then it is standard that
\[
\dim_\textup{H} F = \liminf_{k \to \infty} N(k)/k
\]
and
\[
\dim_\textup{P} F = \limsup_{k \to \infty} N(k)/k.
\]
Let $k_n$ be a very rapidly increasing sequence of integers, e.g. $k_n = n^n$. Fix an integer  $t \geq 2$ and construct $F$ inductively by using the first IFS  $k_n$ times, followed by the second IFS $(t-1)k_n$ times.  This is repeated in blocks of $tk_n$.  Then
\[
\dim_\textup{H} F = 1/t < 1 = \dim_\textup{P} F.
\]
Consider $x \in F$ and especially the cylinder $E$ of level
\[
M=t \sum_{n=1}^{N} k_n 
\]
containing $x$. (Just after the $N$th long run of the second IFS).   This cylinder has size $2^{-M} \leq 2^{-tk_N}$.  There are very large holes to the left and right of this cylinder, of size
\[
>  2 \cdot 2^{-k_N\Delta}
\]
where $\Delta>1$ can be chosen arbitrarily provided $N$ is large enough.  Due to the presence of these holes, there must exist $p \in P(q) \cap E$ for some $q$ with $1/q \geq 2^{-k_N\Delta}$.  This ensures that the whole set $F$ is contained in the 
\[
2^{-tk_N} = \left(2^{-k_N\Delta} \right)^{t/\Delta} \leq (1/q)^{t/\Delta}
\]
neighbourhood of $P(q)$ for infinitely many $q$.  This ensures that
\[
F_{t/\Delta} = F
\]
for all $\Delta>1$.
\end{proof}


\subsubsection{Sharpness of lower bound}

Comparing the lower bound from Theorem \ref{main} with the classical
Jarn\'ik--Besicovitch theorem, one might ask if the lower dimension
can be replaced by the Hausdorff dimension, or indeed if a lower bound
of the form $(\dim_\textup{L} F+1)/t$ or $c\dim_\textup{H} F/t$ is
possible for some uniform constant $c$.  The following proposition
shows that these bounds are not possible in general and demonstrates
that $F_t$ ($t >1$) can be very small, even countable, if the abstract
rationals are ``far away'' from the large parts of $F$.

\begin{prop} \label{example}
There exists a compact set $F \subseteq [0,1]$ with $\dim_\textup{H} F = \dim_\textup{P} F >0$ and such that, for some choice of abstract rationals, $F_t$ is countable for all $t>1$ and, for another choice of abstract rationals,
\[
\dim_\textup{H} F_t  >0
\]
for all $t>1$.
\end{prop}

\begin{proof}
Let $F_0$ be the middle third Cantor set (see \cite[Example 4.5]{falconer}) and $F$ be $F_0 \cup C$ where $C$ is the countable set consisting of a single point  in the centre of each complementary interval of $F_0$.  Clearly $\dim_\textup{H} F = \dim_\textup{P} F =s = \log_32>0$.  We choose the first set of abstract rationals such that $\cup_q P(Q) = C$.   More precisely, for $q \geq 1$, let $C_k \subseteq C$ be the set of centres of complementary intervals of length  $3^{-k}$ where $k$ is chosen such that $3^{-(k-1)}/2 \leq 1/q < 3^{-(k-2)}/2$.  Then choose $P(q)$ to be a maximal $(1/q)$-separated subset of $C$ which contains $C_k$.    Note that $P(q)$ is a maximal $(1/q)$-separated subset of $F$ by construction.  Fix $t>1$,  $q \in \mathbb{N}$ and $p \in P(q)$ and let $k$ be as above.  Then
\[
B(p,q^{-t}) \subseteq  B(p, 3^{-t(k-2)}/2^t) = \{p\}
\]
for sufficiently large $q$.  More precisely, once $q$ is large enough so that $k$ is large enough to guarantee
\[
3^{-t(k-2)}/2^t < 3^{-k}/2.
\]
It follows that $F_t = C$ which is countable and, in particular, 
\[
\dim_\textup{H} F_t = \dim_\textup{P} F_t = 0.
\]
For the second choice of abstract rationals, choose $P(q)$ such that it contains a maximal $(1/q)$-separated subset of $F_0$, which we denote by $P_0(q) \subseteq F_0$.  Let $F_{0,t} = (F_0)_t \subseteq F_0$ be the set of points in $F_0$ which are i.o. $t$-approximable by abstract rationals $P_0(q)$.  Then, since $F_0$ is $s$-Ahlfors--David regular, it follows from Theorem \ref{adthm}, which we prove later,  that
\[
\dim_\textup{H} F_t \geq  \dim_\textup{H} F_{0,t} \geq s/t>0
\]
as required.
\end{proof}

The next result is an upgrade of the previous example to show that the lower bound $\dim_\textup{L} F/t$ is sharp in a stronger sense.

\begin{prop} \label{exampleL}
For all $0\leq L\leq H \leq 1$, there exists a compact set $F \subseteq [0,1]$ with
\[
\dim_\textup{L} F = L
\]
and 
\[
\dim_\textup{H} F = \dim_\textup{P} F =H
\]
 such that for some choice of abstract rationals,
\[
\dim_\textup{H} F_t  =L/t = \frac{\dim_\textup{L} F}{t}
\]
for all $t>1$.
\end{prop}

\begin{proof}
This proposition is proved in the same way as Proposition \ref{example} but with the middle third cantor set replaced by a central (self-similar) cantor set of dimension $H$ and with the single points placed in the complementary intervals replaced by scaled copies of a fixed set $E$ with $\dim_\textup{L} E = \dim_\textup{H} E  = L$. The set $E$ could be another central cantor set, for example. The scaling is done such that the diameter of a scaled copy is much smaller than the length of the complementary interval it is placed in.  The abstract rationals are chosen to be contained in the countably many copies of $E$ and within each copy $E'$ of $E$ must be chosen such that 
\begin{equation} \label{nice}
\dim_\textup{H} E'_t = \frac{\dim_\textup{H} E'}{t}.
\end{equation}
The details work the same, but the challenge is in finding a way to choose the set $E$ and abstract rationals to satisfy \eqref{nice}.  We choose $E$ to be an $L$-Ahlfors--David regular set with the abstract rationals chosen as in Proposition \ref{prop:middlelambda}, which we prove later.
\end{proof}

There are many other ways in which one could investigate sharpness of Theorem \ref{main}.  For example, our results above are pointwise in the sense that for a given $t$ a set $F$ is constructed. It may be interesting to consider if one can construct sets witnessing sharpness   for all $t \geq 1$ simultaneously.  Another question of sharpness which we have not answered is if the \emph{Hausdorff} dimension of $F_t$ can be estimated from below in general.

\begin{ques}
Is it true that for all $t \geq 1$
\[
\dim_\textup{H} F_t \geq \frac{\dim_\textup{L} F}{t}?
\]
\end{ques}


\subsection{The Ahlfors--David regular case}

It turns out we can say more when $F$ is an Ahlfors--David regular metric space.  Recall that a compact set $F$ is  $s$-Ahlfors--David regular for some $s \geq 0$  if there exists a constant $C>0$ such that 
\[
r^s/C \leq \mathcal{H}^s(B(x,r)) \leq Cr^s
\]
for all $0<r<1$ and all $x \in F$.  In this case it is not hard to see that
\[
\dim_\textup{L} F = \dim_\textup{H} F = \dim_\textup{P} F =  \overline{\dim}_\textup{B} F = s
\]
and that $F$ is doubling, see \cite[Theorem 6.4.1]{fraser}.

\begin{thm} \label{adthm}
Suppose $F$ is a compact   $s$-Ahlfors--David regular.  For all $t \geq 1$
\[
\frac{s}{t} \leq \dim_\textup{H} F_t  \leq \frac{s+1}{t}.
\]
\end{thm}

\begin{proof}
Given the general bounds, all that remains to establish is
\[
\dim_\textup{H} F_t  \geq \frac{s}{t} 
\]
which we prove now.  The strategy is similar to the general lower bound, except that we need to  control the measure at all scales, not just a sequence of scales. Let $E$ and $\mu$ be as in the proof of Theorem \ref{main}.  Given $r \in (0,q_1^{-t})$, let $k$ be chosen uniquely such that 
\[
q_{k+1}^{-1} \leq r < q_{k}^{-1}.
\]
Since $t\geq 1$ is fixed and the $q_k$ are rapidly increasing, we may assume that either
\begin{equation} \label{case1}
q_{k+1}^{-1} \leq r < q_{k}^{-t} \leq q_{k}^{-1}
\end{equation} 
or
\begin{equation} \label{case2}
q_{k+1}^{-1} < q_{k}^{-t} \leq r  < q_{k}^{-1}.
\end{equation} 
Consider the first case \eqref{case1}.  There is a uniform constant $C_3$ such that an arbitrary ball $B(x,r)$ may intersect $C_3(r/q_{k+1}^{-1})^s$ many $(k+1)$-level cylinders.  This uses the fact that  $P(q_{k+1})$ is $(1/q_{k+1})$-separated and the fact that $F$ is $s$-Ahfors-David regular.  Moreover, each  $(k+1)$-level cylinder carries mass at most
\begin{equation} \label{dunkel}
q_1^{ts}(q_2 \cdots q_{k})^{s(t-1)} (q_{k+1}^{-1})^{s}
\end{equation} 
by following the argument above \eqref{massestimate}.  Therefore
\begin{eqnarray*}
\mu(B(x,r)) & \leq & C_3(r/q_{k+1}^{-1})^s q_1^{ts}(q_2 \cdots q_{k})^{s(t-1)} (q_{k+1}^{-1})^{s} \qquad \text{by \eqref{dunkel}} \\ \\
& = & C_3q_1^{ts}(q_2 \cdots q_{k})^{s(t-1)}\, r^s  \\ \\
& \leq & C_3 \left( \log(q_k)\right)^{ts} q_{k}^{s(t-1)} \,  r^s \qquad \text{by \eqref{rapid}} \\ \\
& \leq & C_3  |\log r |^{ts}  \left(r^{-1/t}\right)^{s(t-1)} \,  r^s  \qquad \text{by \eqref{case1}} \\ \\
& = & C_3  |\log r |^{ts}  r^{s/t}.
\end{eqnarray*}
Now consider the second case \eqref{case2}.  Since $F$ is Ahlfors--David regular it is doubling and, therefore,  there exists a uniform constant $C_4$ such that an arbitrary ball $B(x,r)$ intersects at most $C_4$ many $k$-level cylinders.   This uses the fact that the centres of the  $k$-level cylinders are $(1/q_k)$-separated.  Therefore, using \eqref{dunkel} to estimate the mass of a $k$-level cylinder,
\begin{eqnarray*}
\mu(B(x,r)) &\leq&  C_4q_1^{ts}(q_2 \cdots q_{k-1})^{s(t-1)} (q_k^{-1})^{s} \\ \\
&\leq & C_4 \left(\log(q_k)\right)^{ts} (q_k^{-1})^{s}  \qquad \text{by \eqref{rapid}}
\\ \\
&\leq  & C_4 |\log r|^{ts}  r^{s/t}  \qquad \text{by \eqref{case2}.}
\end{eqnarray*}
Combining the two cases, we have proved that for an arbitrary $x \in E$ and $r \in (0, q_1^{-t})$
\[
\mu(B(x,r)) \leq \max\{C_3,C_4\} |\log r|^{ts}  r^{s/t}
\]
and  the result follows by the mass distribution principle, see \cite[Chapter 4]{falconer}.
\end{proof}

\section{Dirichlet spectrum of Ahlfors--David regular spaces}\label{sec:spectrum}

Notice that in the classical setting of Diophantine approximation on
$[0,1]$ with $P$ defined by
\begin{equation*}
P(q) = \set*{0, 1/q, 2/q, \dots, 1}, \qquad (q\geq 1),
\end{equation*}
we have $s = 1$ and the \emph{upper} bound in Theorem~\ref{adthm}
gives the correct dimension of $F_t$, namely,
$\dim_{\textup H} F_t = 2/t$, the dimension given in the
Jarn\'ik--Besicovitch theorem. Therefore probing sharpness of the
lower bound seems like the most pressing concern.

Notice also that in the classical setting, we have
\begin{gather*}
  2 = \sup\set*{d : F_d = [0,1]} = \inf\set*{d : \Leb(F_d) = 0} \\
  = \sup\set*{d : F_d\textrm{ contains a ball}} = \sup\set*{d :
    F_d\textrm{ essentially contains a ball}}.
\end{gather*}
As a generalisation of this, define the following: for a given $s$-Ahlfors--David
regular set $F$ and choice $P$ of abstract rationals in $F$, define the
\emph{Dirichlet exponent}
\begin{align*}
  d(P) &= \sup\set*{d : F_d \textrm{ contains a ball up to $\mathcal H^s$-measure $0$}} \\
       &= \sup\set*{d : (\exists \textrm{ball } B\subset F)\quad \mathcal H^s(B\setminus F_d) = 0} 
\end{align*}
and the \emph{Dirichlet spectrum}
\begin{equation*}
  \sigma(F) = \set*{d(P) : P \textrm{ is a choice of abstract rationals in } F}.
\end{equation*}
We can then ask the following.

\begin{ques}\label{q:jarnikTF}
  For a given $s$-Ahlfors--David-regular $F$, what is the nature of
  the set $\sigma(F)$, and to what extent does the formula
  \begin{equation*}
    \dim_{\textup H} F_t = \frac{d(P)s}{t}\qquad(t\geq d(P))
  \end{equation*}
  hold for a given choice $P$ of abstract rationals in $F$?
\end{ques}

In exploring this question, we will find, in particular, that the
dimension bounds in Theorem~\ref{adthm} are sharp. First, let us report
what Theorem~\ref{adthm} and the Mass Transference Principle say
immediately about Question~\ref{q:jarnikTF}.

\begin{proposition}\label{prop:spectrum}
  If $F$ is a locally compact, $s$-Ahlfors--David-regular metric space, then
  $\sigma(F) \subset [1, 1 + 1/s]$, and for any choice $P$ of
  abstract rationals in $F$ we have 
  \begin{equation*}
    \dim_{\textup H}F_t \geq \frac{d(P)s}{t}\qquad (t \geq d(P)).
  \end{equation*}
\end{proposition}

\begin{proof}
  Clearly, since for any $P$ one has $F=F_1$, we have $d(P) \geq
  1$. On the other hand, suppose $\mathcal H^s(F_d)>0$. Then
  $\dim_{\textup H} F_d \geq s$, which, by Theorem~\ref{adthm},
  implies that $s \leq (s+1)/d$, hence $d \leq 1 + 1/s$. Taking
  supremum over all such $d$ shows that $d(P) \leq 1 + 1/s$, so we
  have shown $d(P)\in [1, 1 + 1/s]$ for all $P$.

  Now suppose $d < d(P) \leq t$. It follows that $F_d$
    essentially contains a closed ball $F'=B(x',r) \subset F$, and we
    may suppose that $\mathcal H^s(S(x',r))=0$, where $S(x',r)$
    denotes the sphere of radius $r$ centered at $x'$. (After all, at
    most countably many members of the uncountable disjoint union
    $F'=\bigcup_{0\leq \rho \leq r} S(x',\rho)$ may have positive
    measure, so if we must, we may shrink $r$ until $S(x',r)$ has zero
    measure.) The ball $F'$ is itself an $s$-Ahlfors--David-regular
    metric space. Define a hierarchy of points $P' = P\mid_{F'}$ in
    $F'$ by
  \begin{equation*}
    P'(q) =
    \begin{cases}
      P(q)\cap F'&\textrm{if } P(q)\cap F'\neq \emptyset\\
            x' &\textrm{if } P(q)\cap F' = \emptyset.
    \end{cases}
  \end{equation*}
  Note that as soon as $q$ is large enough we have
  $P'(q) = P(q)\cap F'$, hence $P'(q)$ is $(1/q)$-separated but it may
  not be maximal. That is, $P'$ may not constitute abstract rationals
  in $F'$ as we have defined the term. Still, let us use $F_t'$ to
  denote all the points in $F'$ that are $t$-approximable by
  $P'$. Notice then that
  \begin{equation}\label{eq:insphere}
    (F_d\cap F')\setminus F_d' \subset S(x',r),
  \end{equation}
  for, if $x\in F_d$ and $d(x',x)< r$, then for all $q$ large enough,
  every $p\in P(q)$ with $d(x,p)\leq q^{-d}$ must lie in $F'$. All
  such $p$ are in $P'(q)$, so we have $x\in F_d'$. Importantly for
  us,~(\ref{eq:insphere}) implies that
  $\mathcal H^s(F'\setminus F_d')=0$. 

For any $t\geq d(P)$, we have $F_t' \subset F_t$. Also, since
$\mathcal H^s(F'\setminus F_d')=0$, we have that for any ball
$B\subset F'$,
  \begin{equation*}
    \mathcal H^s (F_d'\cap B) = \mathcal H^s(B).
  \end{equation*}
  So the mass transference principle~\cite[Theorem 3]{BV06} implies
  that for any ball $B\subset F'$ we have
  \begin{equation*}
    \mathcal H^{ds/t} (F_t'\cap B) = \mathcal H^{ds/t}(B) = \infty. 
  \end{equation*}
  Therefore $\dim_{\textup H}F_t' \geq ds/t$, hence
  $\dim_{\textup H} F_t \geq ds/t$. Since $d<d(P)$ was arbitrary, the
  result follows.
\end{proof}

\begin{remark}
  We could of course ask Question~\ref{q:jarnikTF} when $F$ is not
  Ahlfors--David regular. There, Theorem~\ref{main} would guarantee that
  \begin{equation*}
    1 \leq d(P) \leq \frac{\dim_{\textup P} F + 1}{\dim_H F}. 
  \end{equation*}
  However, clearly no such dimension formula as the one in
  Question~\ref{q:jarnikTF} holds in general, even with $s$ replaced
  by $\dim_{\textup{H}}F.$ 
\end{remark}

We can now refine Question~\ref{q:jarnikTF} by asking what kind of
subset $\sigma(F)\subset [1, 1+1/s]$ is. The next results exhibit
cases where $\sigma(F) = [1, 1 + 1/s]$. We will also show that for
certain spaces $F$ and any $d\in \sigma(F)$, there exist abstract
rationals witnessing $d(P)=d$ such that $\dim_{\textup H}F_t = ds/t$
for all $t\geq d$. That is, the full spectrum $\sigma(F)$ can be
realized by abstract rationals satisfying a general
Jarnik--Besicovitch theorem.

In particular, we show that the bounds from Theorem \ref{adthm} are
indeed sharp and, moreover, the dimensions of $F_t$ can depend on the
choice of abstract rationals, even in the Ahlfors--David regular
setting.

\begin{proposition}\label{adsharp}
  For $F=[0,1]$ we have $\sigma(F)= [1,2]$ and for any $d\in\sigma(F)$
  there is a choice $P$ of abstract rationals with $d(P)=d$ such that
  \[
    \dim_{\textup H}F_t = \frac{d(P)}{t} \qquad (t\geq d(P)).
  \]
\end{proposition}

\begin{proof}
  The case $d(P)=2$ is achieved by the classical setting of the
  Jarn\'ik--Besicovitch theorem so it remains to consider
  $d \in [1,2)$.  We first prove the case $d=1$, which provides a
  roadmap to the other cases.

  For each $q\in [2^k, 2^{k+1})$ let
  \begin{equation*}
    P(q) = \set*{\frac{a}{2^k} : a = 0, \dots, 2^k}
  \end{equation*}
  be the rationals with denominator $2^k$. Note that $P(q)$ is a
  maximal $(1/q)$-separated set in $F$. In this setting we have
  $d(P) = 1$, which in particular shows $1\in \sigma(F)$. This will follow from Proposition~\ref{prop:spectrum} together with the upper bound for the dimension of $F_t$ which we prove now.  The lower bound 
  \begin{equation*}
    \dim_{\textup H} F_t \geq \frac{1}{t} \qquad (t\geq 1),
  \end{equation*}
  already follows from Theorem \ref{adthm}.  Note that
  \begin{equation*}
    F_t \subset \bigcup_{k\geq 1} \bigcup_{a=0}^{2^k}
    \brackets*{\frac{a}{2^k} - \frac{1}{2^{tk}}, \frac{a}{2^k} + \frac{1}{2^{tk}}}.
  \end{equation*}
  For all $\tau>1/t$ we have that
  \begin{equation*}
    \sum_{k\geq 1} \sum_{a=0}^{2^k} \parens*{\frac{1}{2^{tk-1}}}^\tau
    \ll \sum_{k\geq 1} 2^{(1-\tau t)k},
  \end{equation*}
  which converges. (The notation $\ll$ means that $\leq$ holds up to
  a positive absolute constant factor.) This implies
  $\mathcal H^\tau(F_t)< \infty$ and $\dim_{\textup H}F_t \leq
  \tau$. Therefore
  $\dim_{\textup H}F_t \leq 1/t$. \\

  Now suppose $d\in(1,2)$ and let $\alpha>1$ be such that
  $d=1+1/\alpha$.  For $q\in\NN$ with
  $k^\alpha \leq q < (k+1)^\alpha$, let
  \begin{equation*}
    P(q) = \set*{\frac{a}{\floor{k^\alpha}} : a = 0, \dots, \floor{k^\alpha}}.
  \end{equation*}
  We claim that $d(P)=d$. In fact, this is a consequence of the
  recently proved Duffin--Schaeffer conjecture~\cite{DSC}, which states that if
  $\psi(n):\NN\to\RR$ is such that
  \begin{equation*}
    \sum_q \psi(q)\varphi(q)=\infty, 
  \end{equation*}
  then for almost every $x\in[0,1]$ there are infinitely many $p/q\in \QQ$ such that
  \begin{equation*}
    \abs*{x - \frac{p}{q}} < \psi(q). 
  \end{equation*}
 Here and throughout $\varphi$ is the is the Euler totient function. Consider the function
  \begin{equation*}
    \psi(q) =
    \begin{cases}
      q^{-d}&\textrm{if } q = \floor*{k^\alpha}, k \in \NN \\
      0 &\textrm{otherwise.}
    \end{cases}
  \end{equation*}
  Then
  \begin{align*}
    \sum_q \psi(q)\varphi(q) &= \sum_k \frac{1}{\floor{k^\alpha}^d}\varphi (\floor{k^\alpha}) \\
                             &\geq \sum_k \frac{\varphi (\floor{k^\alpha})}{k^{\alpha +1}} \\
                             &\geq \sum_k \frac{\varphi (\floor{k^\alpha})}{k^{\alpha +1}} \\
                             &\gg \sum_k \frac{1}{k\log\log k^\alpha} = \infty,
  \end{align*}
  this last estimate coming from the growth estimate
  $\varphi(n) \gg n/\log\log n$~\cite{HW}. We may then conclude that
  for almost every $x\in[0,1]$ there are infinitely many $q$ such that
  $x$ is within $q^{-d}$ of some $p\in P(q)$. That is, $F_d$ has full
  measure, hence $d(P) \geq d$. On the other hand, if we had defined
  $\psi$ by $q^{-d'}$ for some $d'> d$, then the above sum would have
  converged, and the Borel--Cantelli lemma would imply that $F_{d'}$
  had measure $0$. This shows that $d(P)=d$, as claimed. 

  Proposition~\ref{prop:spectrum} now gives
  \begin{equation*}
    \dim_{\textup H}F_t \geq \frac{d(P)}{t}
  \end{equation*}
  for every $t\geq d$. The corresponding upper bound comes from a
  cover of $F_t$ by intervals as in the previous case. Here the
  relevant series is
  \begin{equation*}
    \sum_{k\geq 1}\sum_{a=0}^{\floor{k^\alpha}}\parens*{k^{-\alpha t}}^\tau
    \ll \sum_{k\geq 1} k^{\alpha (1-t\tau)}, 
  \end{equation*}
  which converges as long as
  \begin{equation*}
    \tau > \frac{1 +
      1/\alpha}{t} = \frac{d}{t},
  \end{equation*}
  which leads to the upper bound and establishes
  \begin{equation}\label{eq:1}
    \dim_{\textup H}F_t = \frac{1 +
      1/\alpha}{t} = \frac{d(P)}{t}.
  \end{equation}
  Since $d\in(1,2)$ was arbitrary, we are done.
\end{proof}

We investigate Dirichlet exponents for examples of Ahlfors--David
regular $F$ with dimension less than 1.

\begin{proposition}\label{adexample}
  Let $F$ be the middle third Cantor set, whose dimension is
  $s=\log_3 2$.  Then for any $d\in\brackets*{1, 1 + 1/s}$
  there exists a choice of abstract rationals such that $d(F)=d$ and
  \[
    \dim_{\textup{H}}F_t =\frac{d}{t}s \quad (t\geq d).
  \]
\end{proposition}

\begin{proof}
  (i) \emph{The case $d=1$.}  First, for $3^n \leq q < 3^{n+1}$ let $P(q)$ be the endpoints of
  the basic intervals in the $n$th construction set
  $E_n$. Theorem~\ref{adthm} tells us that for any $t\geq 1$,
  $\dim_HF_t \geq s/t$. For the upper bound, note that $F_t$ is
  covered by
  \begin{equation*}
    F_t \subset \bigcup_{n \geq m} \bigcup_{p\in P(3^n)}
    \brackets*{p-3^{-nt}, p + 3^{-nt}}
  \end{equation*}
  for any $m>0$. For $\tau\geq 0$, the sum of diameters is
  \begin{equation*}
    \sum_{n \geq m} 2^{n+2} 3^{-nt\tau},
  \end{equation*}
  which converges as long as $2 < 3^{t\tau}$, that is, if
  $\tau > s/t$. This shows the other inequality in dimension, hence
  $\dim_H F_t = s/t$.

  Notice that in this example we have that for every $x\in F$ there
  are infinitely many $q$ for which there is $p\in P(q)$ satisfying
  $\abs{x-p} \leq 1/(3q)$. The previous paragraph shows that as soon
  as we increase the exponent beyond $1$, the dimension
  drops. Therefore, the $d(F) = 1$ with this choice of $P$.\\

  (ii)  \emph{The case $d=1+1/s$. }  For $n\geq 0$, let
  \begin{equation*}
    k(n) = k = \floor*{1 + n\parens*{1+\frac{1}{s}}},
  \end{equation*}
  that is, $k(n)$ is the greatest integer with the property that each
  basic interval of $E_n$ contains at most $3^{n+1}-3^n$ basic
  intervals of $E_k$.  Denote by
  \begin{equation*}
    E_n^L = \bigcup_{\substack{I \subset E_n \\ \textrm{basic
    interval}}} \set*{\min I } = \set{\ell_1, \dots, \ell_{2^n}}
  \end{equation*}
  the set of left endpoints $\ell_i$ of the basic intervals of $E_n$. For
  $3^n \leq q < 3^{n+1}$ put
  \begin{equation*}
    m = \parens*{q-3^n \quad (\bmod 2^{k-n})} + 1
  \end{equation*}
  and define
  \begin{equation*}
    P(q) = \parens*{E_n^L + \ell_m}\cup \set*{\textrm{enough endpoints of $E_n$ to get maximality}}.
  \end{equation*}
  In particular, $P(3^n)$ consists exactly of the endpoints of the
  basic intervals of $E_n$, as in the previous example. (See
  Figure~\ref{fig:cantorrationals} for an illustration of $P(1)$
  throught $P(9)$.) Importantly, we have, for $3^n \leq q < 3^{n+1}$
  \begin{equation*}
    2^{n-1} = \frac{\#P(3^n)}{2}\leq \#P(q) \leq \#P(3^n) = 2^n.
  \end{equation*}

  The ``Dirichlet theorem'' here is: For every $x\in F$ and $n\in\NN$,
  there exists $q\leq 3^{n+1}$ and $p\in P(q)$ such that
  \begin{equation*}
    \abs*{x- p} \leq 3^{-k} \leq \frac{1}{q^{k/(n+1)}}.
  \end{equation*}
  Since
  \begin{equation*}
    \frac{k(n)}{n+1} \sim 1 + \frac{1}{s},
  \end{equation*}
  the Dirichlet exponent here is $d(F) = 1 + 1/s$.

  By Proposition~\ref{prop:spectrum} we have
  $\dim_H F_t \geq d(P) s/t = \frac{s + 1}{t}$ for all $t\geq
  d(P)$. For the upper bound, we cover
  \begin{equation*}
    F_t \subset \bigcup_{q\geq 3^m} \bigcup_{p\in P(q)} B(p, q^{-t})
    \subset \bigcup_{q\geq 3^m} \bigcup_{p\in P(q)} B(p, 3^{-nt})
  \end{equation*}
  The corresponding sum of diameters, with $\tau>0$, is
  \begin{equation*}
    \sum_{q\geq 3^m} \sum_{p \in P(q)} 3^{-nt\tau} = \sum_{n\geq m} 2^{k}
    3^{-nt\tau} \leq \sum_{n\geq m} 2^{n(1 + 1/s)}
    3^{-nt\tau} ,
  \end{equation*}
  which converges if $\tau>(s + 1)/t$. This leads to the desired
  upper bound, so we have
  \begin{equation*}
    \dim_H F_t = \frac{s + 1}{t} = \frac{d(F)}{t}s.
  \end{equation*}

    \begin{figure}\label{fig:cantorrationals}
      \begin{tikzpicture}


        \node at (-8, 4) {$P(1)$};

        \draw[ultra thick, |-|] (-6, 4) -- (6, 4);

        \fill[red] (-6, 4) circle (4pt);

        \fill[blue] (6, 4) circle (4pt);

        
        \node at (-8, 3){$P(2)$};

        \draw[ultra thick, |-|] (-6, 3) -- (-2, 3);

        \draw[ultra thick, |-|] (2, 3) -- (6, 3);

        \fill[red] (2, 3) circle (4pt);

        \fill[blue] (-6, 3) circle (4pt);

        
        \node at (-8, 2){$P(3)$};

        \draw[ultra thick, |-|] (-6, 2) -- (-2, 2);

        \draw[ultra thick, |-|] (2, 2) -- (6, 2);

        \fill[red] (-6,2) circle (4pt);

        \fill[blue] (-2,2) circle (4pt);

        \fill[red] (2,2) circle (4pt);

        \fill[blue] (6,2) circle (4pt);


        \node at (-8, 1){$P(4)$};

        \draw[ultra thick, |-|] (-6, 1) -- (-50/9, 1);

        \draw[ultra thick, |-|] (-46/9, 1) -- (-42/9, 1);

        \draw[ultra thick, |-|] (-10/3, 1) -- (-26/9, 1);

        \draw[ultra thick, |-|] (-22/9, 1) -- (-18/9, 1);

        \draw[ultra thick, |-|] (6, 1) -- (50/9, 1);

        \draw[ultra thick, |-|] (46/9, 1) -- (42/9, 1);

        \draw[ultra thick, |-|] (10/3, 1) -- (26/9, 1);

        \draw[ultra thick, |-|] (22/9, 1) -- (18/9, 1);

        \fill[red] (-46/9, 1) circle (4pt);

        \fill[red] (26/9, 1) circle (4pt);

        \fill[blue] (-2, 1) circle (4pt);

        \fill[blue] (6, 1) circle (4pt);


        \node at (-8, 0){$P(5)$};

        \draw[ultra thick, |-|] (-6, 0) -- (-50/9, 0);

        \draw[ultra thick, |-|] (-46/9, 0) -- (-42/9, 0);

        \draw[ultra thick, |-|] (-10/3, 0) -- (-26/9, 0);

        \draw[ultra thick, |-|] (-22/9, 0) -- (-18/9, 0);

        \draw[ultra thick, |-|] (6, 0) -- (50/9, 0);

        \draw[ultra thick, |-|] (46/9, 0) -- (42/9, 0);

        \draw[ultra thick, |-|] (10/3, 0) -- (26/9, 0);

        \draw[ultra thick, |-|] (22/9, 0) -- (18/9, 0);

        \fill[red] (-10/3, 0) circle (4pt);

        \fill[red] (42/9, 0) circle (4pt);

        \fill[blue] (-6, 0) circle (4pt);

        \fill[blue] (2, 0) circle (4pt);


        \node at (-8, -1){$P(6)$};

        \draw[ultra thick, |-|] (-6, -1) -- (-50/9, -1);

        \draw[ultra thick, |-|] (-46/9, -1) -- (-42/9, -1);

        \draw[ultra thick, |-|] (-10/3, -1) -- (-26/9, -1);

        \draw[ultra thick, |-|] (-22/9, -1) -- (-18/9, -1);

        \draw[ultra thick, |-|] (6, -1) -- (50/9, -1);

        \draw[ultra thick, |-|] (46/9, -1) -- (42/9, -1);

        \draw[ultra thick, |-|] (10/3, -1) -- (26/9, -1);

        \draw[ultra thick, |-|] (22/9, -1) -- (18/9, -1);

        \fill[red] (-22/9, -1) circle (4pt);

        \fill[red] (50/9, -1) circle (4pt);

        \fill[blue] (-6, -1) circle (4pt);

        \fill[blue] (2, -1) circle (4pt);


        \node at (-8, -2){$P(7)$};

        \draw[ultra thick, |-|] (-6, -2) -- (-50/9, -2);

        \draw[ultra thick, |-|] (-46/9, -2) -- (-42/9, -2);

        \draw[ultra thick, |-|] (-10/3, -2) -- (-26/9, -2);

        \draw[ultra thick, |-|] (-22/9, -2) -- (-18/9, -2);

        \draw[ultra thick, |-|] (6, -2) -- (50/9, -2);

        \draw[ultra thick, |-|] (46/9, -2) -- (42/9, -2);

        \draw[ultra thick, |-|] (10/3, -2) -- (26/9, -2);

        \draw[ultra thick, |-|] (22/9, -2) -- (18/9, -2);

        \fill[red] (-6, -2) circle (4pt);

        \fill[red] (2, -2) circle (4pt);

        \fill[blue] (-2, -2) circle (4pt);

        \fill[blue] (6, -2) circle (4pt);


        \node at (-8, -3){$P(8)$};

        \draw[ultra thick, |-|] (-6, -3) -- (-50/9, -3);

        \draw[ultra thick, |-|] (-46/9, -3) -- (-42/9, -3);

        \draw[ultra thick, |-|] (-10/3, -3) -- (-26/9, -3);

        \draw[ultra thick, |-|] (-22/9, -3) -- (-18/9, -3);

        \draw[ultra thick, |-|] (6, -3) -- (50/9, -3);

        \draw[ultra thick, |-|] (46/9, -3) -- (42/9, -3);

        \draw[ultra thick, |-|] (10/3, -3) -- (26/9, -3);

        \draw[ultra thick, |-|] (22/9, -3) -- (18/9, -3);

        \fill[red] (-46/9, -3) circle (4pt);

        \fill[red] (26/9, -3) circle (4pt);

        \fill[blue] (-2, -3) circle (4pt);

        \fill[blue] (6, -3) circle (4pt);


        \node at (-8, -4){$P(9)$};

        \draw[ultra thick, |-|] (-6, -4) -- (-14/3, -4);

        \draw[ultra thick, |-|] (6, -4) -- (14/3, -4);

        \draw[ultra thick, |-|] (10/3, -4) -- (2, -4);

        \draw[ultra thick, |-|] (-10/3, -4) -- (-2, -4);

        \fill[red] (2, -4) circle (4pt);

        \fill[red] (-6, -4) circle (4pt);

        \fill[blue] (-14/3, -4) circle (4pt);

        \fill[red] (14/3, -4) circle (4pt);

        \fill[red] (-10/3, -4) circle (4pt);

        \fill[blue] (-2, -4) circle (4pt);

        \fill[blue] (6, -4) circle (4pt);

        \fill[blue] (10/3, -4) circle (4pt);

        \fill[red] (2, -4) circle (4pt);


        \node at (0, -5){$\vdots$};

        \node at (-4, -5){$\vdots$};

        \node at (-8, -5){$\vdots$};

        \node at (4, -5){$\vdots$};

      \end{tikzpicture}
      \caption{The red dots are the elements of $\parens*{E_n^L+ \ell_m}$ and the
        blue dots are the endpoints of $E_n$ we have included to
        achieve maximality. United they are $P(q)$.}
    \end{figure}
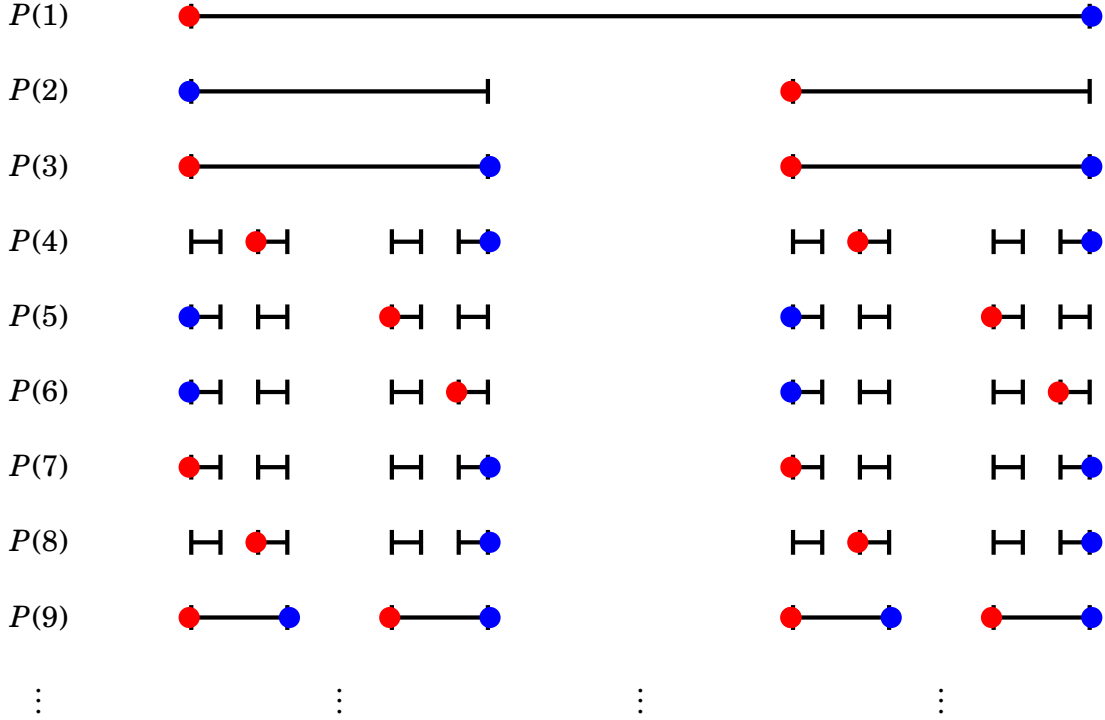

    (iii) \emph{The intermediate case.} Let $d\in (1, 1 + 1/s)$. We now modify the previous
    construction by using instead
    \begin{equation*}
      k(n) = k = \floor*{nd}
    \end{equation*}
    and keeping the rest of the definitions intact. The same Dirichlet
    theorem from the previous construction continues to hold, but now we have
    \begin{equation*}
      \frac{k(n)}{n+1} \sim d,
    \end{equation*}
    so that $d(F) \geq d$.

    For the other bound, let $d' > d$, and note that $F_{d'}$ is covered by
    \begin{equation*}
      F_{d'} \subset \bigcup_{q\geq q_0}\bigcup_{p\in P(q)}B\parens*{p, q^{-d'}}
    \end{equation*}
    for any $q_0$, and therefore the measure of $F_{d'}$ is bounded by
    \begin{align*}
      \sum_{q\geq q_0}\sum_{p\in P(q)} q^{-d's} &\leq \sum_{n\geq n_0}\sum_{3^n\leq q < 3^{n+1}}\sum_{p\in P(q)}q^{-d's}\\
                                                &\leq \sum_{n\geq n_0}\sum_{3^n\leq q < 3^{n+1}}2^n 3^{-nd's}\\
                                                &\leq \sum_{n\geq n_0}2^{k-n}2^n 3^{-nd's}\\
                                                &= \sum_{n\geq n_0}3^{(k-nd')s}. 
    \end{align*}
    This sum converges because $k/n\sim d$ and $d'>d$, so by sending
    $n_0\to \infty$ we see that $F_{d'}$ has zero measure. Therefore
    $d(F) = d$ as wanted.

    Let $t \geq d$. Then, by Proposition~\ref{prop:spectrum}
    $\dim_{\textup H} F_t \geq \frac{ds}{t}$. For the upper upper bound we cover
    \begin{equation*}
      F_t \subset \bigcup_{n\geq 1}\underbrace{\bigcup_{3^n\leq q\leq 3^{n+1}}}_{\textrm{redundant}}\bigcup_{p\in P(q)} B(p, q^{-t})
    \end{equation*}
    and observe that the indicated union has a redundancy in it owing
    to the fact that as $q$ ranges through $(3^n, 3^{n+1})\cap \ZZ$, $P(q)$ cycles with period $2^{k-n}$. That is, we may more efficiently cover by
    \begin{equation*}
      F_t \subset \bigcup_{n\geq 1}\bigcup_{q= 3^n}^{3^n +2^{k-n}}\bigcup_{p\in P(q)} B(p, q^{-t})
    \end{equation*}
    The diameter sum is
    \begin{align*}
      \sum_{n\geq 1}\sum_{q= 3^n}^{3^n +2^{k-n}}\sum_{p\in P(q)} \diam\parens*{B(p, q^{-t})}^\tau &\ll       \sum_{n\geq 1}\sum_{q= 3^n}^{3^n +2^{k-n}}\sum_{p\in P(q)} q^{-t\tau} \\
                                                                                                  &=   \sum_{n\geq 1}\sum_{q= 3^n}^{3^n +2^{k-n}}\#P(q)q^{-t\tau} \\
                                                                                                  &\leq   \sum_{n\geq 1}2^{k-n}\#P(3^n)3^{-nt\tau} \\
                                                                                                  &\leq\sum_{n\geq 1}2^k 3^{-nt\tau}\\
                                                                                                        &\leq\sum_{n\geq 1}3^{ks-nt\tau}.
    \end{align*}
    This sum converges if $\tau > ds/t$, therefore
    $\dim_{\textup H} F_t \leq ds/t$ and we are done.
  \end{proof}

The above result holds for more general Ahlfors--David regular sets than just the middle third Cantor set.  Since it is especially useful to have examples with different dimensions having minimal Dirichlet exponent, see Proof of Proposition \ref{exampleL}, we record the following partial extension.

  \begin{proposition}\label{prop:middlelambda}
Let $F$ be the central  Cantor set having dimension
    $s = -\log 2/\log \lambda$, where $\lambda\in(0,1/2)$. Then there is
    a choice $P$ of abstract rationals such that
    $\dim_{\textup H} F_t = s/t$ for all $t\geq 1$.
  \end{proposition}

  \begin{proof}
    The proof of this statement is essentially the same as the proof of Proposition \ref{adexample} (i). Using the analogous notation, for $\lambda^{-n} \leq q  < \lambda^{-n-1}$ let
    \begin{equation*}
      P(q) = E_n^L \cup \set*{\textrm{enough endpoints of $E_n$ to get maximality}}.
    \end{equation*}
    Notice that $2^{n-1} \leq P(q) \leq 2^n$. Notice that $P(\lambda^{-n})=E_n^L$ and every point of $F$ is in the $\lambda^{n}$-neigbourhood of every $P(\lambda^{-n})$, so that $d(F)\ge 1$. 
    
    What is more, for any $m$,
    \[
    F_t\subset \bigcup_{n\ge m}\bigcup_{p\in P(\lambda^{-n})}[p-\lambda^{-nt}, p+\lambda^{-nt}].
    \]
   This implies that 
   \[
   \mathcal H^\tau (F_t) \le \sum_{m\ge n} 2^n \lambda^{-nt\tau},
   \]
    which converges if $\tau>s/t$. This implies two things: for $t>1$, we immediately obtain $\dimH F_t<\dimH F$ so that $d(F)= 1$. Further, we also see that $\dimH F_t\le s/t$. Since Theorem \ref{adthm} tells us that
    $\dim_{\textup H}F_t \geq s/t$, we have finished the proof  of the theorem. 
  \end{proof}


\begin{thebibliography}{99}
\bibitem[AB]{AB}
D. Allen and B. B\'ar\'any. {\em Hausdorff measures of shrinking targets on self-conformal sets.}, preprint,  arXiv:1911.03410. 

\bibitem[B19]{B19}
S. Baker.
{\em An analogue of Khintchine's theorem for self-conformal sets.}
Math. Proc. Cambridge Philos. Soc. 167 (2019), no. 3, 567–597. 


\bibitem[BV06]{BV06}
V. Beresnevich and S. Velani. {\em A mass transference principle and the Duffin-Schaeffer conjecture for Hausdorff measures.} Ann. of Math. (2) 164 (2006), no. 3, 971–992. 

\bibitem[BFR11]{BFR11}
R. Broderick, L. Fishman and A. Reich.
{\em Intrinsic approximation on Cantor-like sets, a problem of Mahler.}
Mosc. J. Comb. Number Theory 1 (2011), no. 4, 3–12. 

\bibitem[BD16]{BD16}
Y. Bugeaud and A. Durand. 
{\em Metric Diophantine approximation on the middle-third Cantor set.}
J. Eur. Math. Soc. (JEMS) 18 (2016), no. 6, 1233–1272. 

\bibitem[BBDF09]{velanibook}
 V. Beresnevich, V. Bernik, M. Dodson and S. Velani.
\emph{Classical metric Diophantine approximation revisited}, Analytic number theory, 38--61, Cambridge Univ. Press, Cambridge,
(2009).

\bibitem[F97]{techniques}
K.~J. Falconer.
 {\em Techniques in Fractal Geometry},
John Wiley, (1997).

\bibitem[F14]{falconer}
K. J. Falconer.
{\em Fractal Geometry: Mathematical Foundations and Applications},
 John Wiley \& Sons, Hoboken, NJ, 3rd. ed., (2014).




\bibitem[FMS18]{FMS18}
L. Fishman, K. Merrill and D. Simmons. 
{\em Uniformly de Bruijn sequences and symbolic Diophantine approximation on fractals. }
Ann. Comb. 22 (2018), no. 2, 271–293. 

\bibitem[FKMS18]{FKMS18}
 L. Fishman, D. Kleinbock, K. Merrill and D. Simmons. {\em Intrinsic Diophantine approximation on manifolds: general theory.} Trans. Amer. Math. Soc. 370 (2018), no. 1, 577–599. 

\bibitem[FS14]{FS14}
L. Fishman and D. Simmons
{\em Intrinsic approximation for fractals defined by rational iterated function systems: Mahler's research suggestion}.
Proc. Lond. Math. Soc. (3) 109 (2014), no. 1, 189–212. 

\bibitem[F20]{fraser}
J. M. Fraser. 
{\em Assouad dimension and fractal geometry}, 
CUP, Tracts in Mathematics Series,  222, (2020).

\bibitem[HW]{HW}
G. H.   Hardy and E. M. Wright
  {\em An Introduction to the Theory of Numbers},
OUP, 6th Edition (2008).

\bibitem[LSV07]{LSV07}
J.  Levesley, C. Salp and S. Velani. {\em On a problem of K. Mahler: Diophantine approximation and Cantor sets.} Math. Ann. 338 (2007), no. 1, 97–118.

\bibitem[KM]{DSC}
 D.  Koukoulopoulos and J. Maynard
  {\em Duffin--Schaeffer Conjecture},
  Ann. Math. 192 (1),  (2020), 251--307. 
  
  
\bibitem[S21]{S21}
J. Schleischitz.
{\em On intrinsic and extrinsic rational approximation to Cantor sets.}
Ergodic Theory Dynam. Systems 41 (2021), no. 5, 1560–1589. 

\end{thebibliography}
\end{document}